\documentclass[a4paper, 11p]{amsart}
\usepackage[english]{babel}
\usepackage{amsmath}
\usepackage{amssymb}
\usepackage{amsthm}
\usepackage{latexsym}
\usepackage[dvipsnames]{xcolor}
\usepackage{mathtools}
\usepackage{thmtools}
\usepackage{amsfonts}
\usepackage[T1]{fontenc}
\usepackage{tikz}
\usetikzlibrary{arrows,matrix}
\usepackage{enumerate}
\usepackage{wasysym}
\usepackage[parfill]{parskip}
\usepackage{fdsymbol}
\usepackage{stmaryrd}

\usepackage[colorlinks]{hyperref}
\hypersetup{
    colorlinks,
    linkcolor={Blue},
    citecolor={Green},
    urlcolor={Magenta}
}

\usepackage[margin=1in]{geometry}

\begingroup
\makeatletter
   \@for\theoremstyle:=definition,remark,plain\do{%
     \expandafter\g@addto@macro\csname th@\theoremstyle\endcsname{%
        \addtolength\thm@preskip\parskip
     }%
   }
\endgroup

\newtheoremstyle{indented}{5pt}{5pt}{\itshape}{2.5em}{\bfseries}{.}{.5em}{}

\theoremstyle{plain}
\newtheorem{theorem}{Theorem}[section]
\newtheorem{lemma}[theorem]{Lemma}
\newtheorem{proposition}[theorem]{Proposition}
\newtheorem{corollary}[theorem]{Corollary}
\newtheorem{definition}[theorem]{Definition}

\theoremstyle{definition}
\newtheorem{example}[theorem]{Example}
\newtheorem{remark}[theorem]{Remark}

\theoremstyle{indented}
\newtheorem*{claim*}{\indent Claim}

\newcommand{\calC}{\mathcal{C}}

\newcommand{\calO}{\mathcal{O}}

\newcommand{\bbA}{\mathbb{A}}

\newcommand{\bbP}{\mathbb{P}}

\newcommand{\Sym}{\mathrm{Sym}}

\renewcommand{\char}{{\rm char}}

\title[The normal map for plane curves and pathologies in positive characteristic]{The normal map for plane curves \\ and pathologies in positive characteristic}
\author{Edoardo Ballico, Alessandro Oneto}

\address[E. Ballico, A. Oneto]{Universit\`a di Trento, Via Sommarive, 14 - 38123 Povo (Trento), Italy}
\email{edoardo.ballico@unitn.it, alessandro.oneto@unitn.it}

\subjclass[2010]{14H50, 14G17, 14N99}
\keywords{normal lines, bottlenecks, strange curves, dual curves}

\begin{document}

\maketitle

\begin{abstract}
	We study the normal map for plane projective curves, i.e., the map associating to every regular point of the curve the normal line at the point in the dual space. We first observe that the normal map is always birational and then we use this fact to show that for smooth curves of degree higher than four the normal map uniquely determines the curve. Our proof works in characteristic zero and in positive characteristic higher than the degree of the curve. We notice also that in high characteristic strange curves provide examples of different plane curves with same curve of normal lines. We will reinterpret our results also in the modern terminology of bottlenecks of algebraic curves.
\end{abstract}

\section{Introduction}
The study of curves in projective space and their \textit{dual curves} parametrizing tangent lines at regular points in the dual space is very classical. In characteristic zero or in characteristic higher than the degree of the curve it is well-known that a \textit{reciprocity theorem} holds, i.e., the double dual of a curve is equal to the curve, see \cite[Theorem 5.1]{ragni}. This is not true in general in positive characteristic. For example, in characteristic $2$ all tangents to a non-degenerate smooth plane conic pass through a unique point which is, therefore, equal to the double dual of the conic. Moreover, in any positive characteristic there exists singular curves, called \textit{strange curves}, whose tangent lines at regular points pass through a unique point and for which again reciprocity does not hold. We refer to \cite{kleiman} for more details.

We consider the \textit{curve of normal lines} in the dual space, i.e., the curve parametrizing normal lines to the curve at regular points and we address the following natural question.
\begin{center}
	\textit{Is a plane curve uniquely determined by its curve of normal lines?}
\end{center}
In Theorem \ref{thm:birational}, we see that the map sending a regular point of an integral plane curve to its normal line in the dual space is birational, i.e., a general normal line is normal only at one point of the curve. By using this fact, in Theorem \ref{thm:equal_curves}, we give an affirmative answer to our question in the case of smooth curves of degree at least four. Our proof works over algebraically closed fields of characteristic zero or characteristic strictly higher than the degree of the curve. In Example \ref{ex:pathology}, we observe that instead strange curves provide examples of different (singular) plane curves with the same curve of normal lines. 

In Section \ref{ssec:normal_maps} we introduce all notation and we formulate our main results. In Section \ref{ssec:bottlenecks}, we rephrase our result from the point of view of \textit{bottlenecks} of algebraic curves, in the terminology of \cite{dew20}. In Section \ref{ssec:separability} we give preliminary results about normal maps and in Section \ref{ssec:strange} we present the pathological examples of strange curves. In Section \ref{sec:proofs} we give the proofs of our main theorems.

\subsection{Normal map and curve of normal lines}\label{ssec:normal_maps}
We work over an algebraically closed field $\Bbbk$. Let $V$ be a $3$-dimensional $\Bbbk$-vector space. We fix a basis for the dual space $V^\vee$ such that the symmetric algebra $\Sym^\bullet V^\vee$ is identified with the polynomial ring $\Bbbk[x,y,z]$. Let $V_\infty \subset V$ be a $2$-dimensional vector subspace and $H_\infty = \bbP V_{\infty} \subset \bbP V$ be the \textit{line at infinity}. Fixing a non-degenerate quadric $Q \in \Sym^2 V_\infty^\vee$, we get a notion of orthogonality on $V_\infty$. For example, let $i^2 = -1$ in $\Bbbk$. If we can consider $H_{\infty} = \{z = 0\}$ and $Q_\infty = \{x^2+y^2=0, z = 0\} = \{(1:\pm i:0)\}$, then we recover the usual \textit{Euclidean geometry}. 

We use the standard notation $\langle a,b \rangle$ to denote the line through two given points $\{a,b\}$.

Let $X \subset \bbP V$ be an integral plane curve defined by a homogeneous polynomial $F \in \Sym^dV^\vee$ of degree $d > 1$. For every point $p = (p_0:p_1:p_2) \in X_{\rm reg}$, let $T_{X,p}$ be the \textit{tangent line} to $X$ at $p$ given by 
\[
	F_x(p)x + F_y(p)y + F_z(p)z = 0
\]
where $(F_x,F_y,F_z) := \left(\partial_xF,\partial_yF,\partial_zF\right)$. The \textbf{normal line} to $X$ at $a \in X_{\rm reg}$ is the line given by 
\[
	N_{X,p} = \langle p, n_{X,p} \rangle, \quad \text{ where }n_{X,p}\text{ is the point }(T_{X,p} \cap H_\infty)^\perp \in H_\infty.
\]
In the dual space $\bbP V^\vee$, let $X^\perp$ be the integral curve given by the Zariski closure of the set of normal lines at general points of $X$. We call it the \textbf{curve of normal lines} of $X$. If $u : C \rightarrow X$ is the normalization of $X$, then the map sending a regular point of $X$ to its normal line induces the \textbf{normal map}
\[
	\eta_X : C \rightarrow X^\perp.
\]
Our main results, whose proofs are given in Section \ref{sec:proofs}, are the following.
\begin{theorem}\label{thm:birational}
	Let $X \subset \bbP^2$ be an integral curve of degree $d > 1$. Assume that $\char(\Bbbk) = 0$ or $\char(\Bbbk) > d$. Then, $\eta_X$ is birational onto its image. 
\end{theorem}
\begin{theorem}\label{thm:equal_curves}
	Let $X, Y$ be smooth plane curves both of degree larger than $4$. Assume that $\char(\Bbbk) = 0$ or $\char(\Bbbk) > \max\{\deg(X), \deg(Y)\}$. If $X^\perp = Y^\perp$, then $X = Y$. 
\end{theorem}

In Example \ref{ex:pathology}, we see that Theorem \ref{thm:equal_curves} does not hold in general. Indeed, in any positive characteristic we can find pairs of different (singular) curves of degree larger than the characteristic having the same curve of normal lines. This pathological behaviour is again observed within strange curves. 

\subsection{Bottlenecks of curves}\label{ssec:bottlenecks}
We rephrase our results in the terminology of \textit{bottlenecks} of \cite{dew20}. 

Given an integral curve $X$ in $\bbP V$, a line $\ell$ is a \textit{bottleneck} for $X$ if $\ell = N_{X,a} = N_{X,b}$ for $a \in X_{\rm reg}, b \in X_{\rm reg}$. Let $B(X) \subset \bbP V^\vee$ be the \textit{closure of the set of bottlenecks}. Theorem \ref{thm:birational} says that~for every plane curve the normal line at a general point of the curve is not a \textit{bottleneck}, i.e., it is not normal to the curve at another point. In other words, if $\char(\Bbbk) = 0$ or $\char(\Bbbk) > \deg(X)$, then $B(X)$ is a finite set of points.

In Example \ref{ex:pathology}, we notice that this is no longer the case if we consider reducible curves. Indeed, the notion of bottlenecks can be naturally extended to the case of pairs of integral curves or, more in general, for pairs of equidimensional varieties. Given two integral curves $X$ and $Y$ in $\bbP V$, we call \textit{bottleneck} between $X$ and $Y$ a line $\ell$ such that $\ell = N_{X,a} = N_{Y,b}$ where $a \in X_{\rm reg}, b \in Y_{\rm reg}$. We consider the closure $B(X,Y) \subseteq \bbP V^\vee$ of the set of bottlenecks between $X$ and $Y$ in the dual space. For higher equidimensional varieties, we would consider this definition in the Grassmannian of linear spaces of complementary dimension. Obviously, $B(X,Y) \subseteq X^\perp \cap Y^\perp$ and $B(X \cup Y) = B(X,Y) \cup B(X) \cup B(Y)$.

In any odd positive characteristic, Example~\ref{ex:pathology} provides pairs of different curves $X$ and $Y$ having the same curve of normal lines. In particular, they are pairs of strange curves with same strange point and for which $B(X,Y) = B(X \cup Y) = X^\perp = Y^\perp$ is the line dual to the unique strange point.

From Theorem \ref{thm:equal_curves}, we deduce that the latter pathology does not appear if we consider smooth curves. In particular, if we consider $X$ and $Y$ two different integral smooth plane curves such that $\char(\Bbbk) = 0$ or $\char(\Bbbk) > \max\{\deg(X), \deg(Y)\}$ then $X^\perp \neq Y^\perp$ and their set-theoretic intersection in a set of points. In particular, $B(X,Y)$ is a set of points. Since, by Theorem \ref{thm:birational}, $B(X)$ and $B(Y)$ are also sets of points, we deduce that $B(X \cup Y)$ is a finite set.

\section{Preliminaries}

\subsection{Separability of normal maps}\label{ssec:separability}
Consider the \textit{conormal variety} of $X \subseteq \bbP V$
\[
	\calC_X = \overline{\{(p,H) ~:~ p \in X_{\rm reg}, H \supset T_{X,p}\}} \subseteq \bbP V \times \bbP V^\vee.
\]
Consider the projection on the second factor
\[
	\pi_X : \calC_X \rightarrow \bbP V^\vee, \quad (p,H) \mapsto H.
\]
The variety $X^\vee := \pi_X(\calC_X)$ is called the \textit{dual variety} of $X$. A variety $X \subseteq \bbP V$ is called \textit{reflexive} if $C_X = C_{X^\vee}$, which implies that $(X^\vee)^\vee = X$. It is classically known that in characteristic zero, any variety is reflexive (see \cite[Section 16.20]{harris}. Examples of non-reflexive varieties in positive characteristic are given by \textit{strange curves}, i.e., curves whose tangent lines at regular points pass through a given point (see \cite[Section IV.3]{hart}). The dual curve of a strange curve is a hyperplane whose dual is a point. An important result is the \textit{Monge-Segre-Wallace criterion} saying that a variety $X$ is reflexive if and only if the map $\pi_X$ is generically smooth or, equivalently, \textit{separable} (see \cite[page~169]{kleiman}). 

Since $X$ is assumed to be an integral plane curve, the map $\pi_X$ is birational to the rational map sending a regular point of $X$ to its tangent space, which can be extended to the \textbf{tangent map}
\[
	\tau_X : C \rightarrow X^\vee,
\]
where $u : C \rightarrow X$ is again the normalization of $X$.
\begin{lemma}\label{lemma:separability}
	The map $\tau_X$ is separable if and only if $\eta_X$ is separable. I.e., $X$ is reflexive if and only if $\eta_X$ is separable.
\end{lemma}
\begin{proof}
	Recall that in order to prove that a morphism is separable, it is enough to prove that the differential map is surjective on a smooth point (see e.g. \cite[Proposition 6.2.27]{perrin}). Fix $p \in X_{\rm reg}$ a general point of $X$. Consider a system of coordinates such that $p = (0,0) \in \bbA^2$. Since $X$ is smooth at $p$, the completion of the local ring is $\hat{\calO}_{X,p} \simeq \Bbbk\llbracket t\rrbracket$. Hence, in the formal variable $t$, $X_{\rm reg}$ is locally defined in $p$ by the image of a map $t \mapsto (x(t),y(t))$ where $x,y \in \Bbbk\llbracket t\rrbracket$. Hence, we have that locally in $p$ the tangent map $\tau_X$ is the map $t \mapsto (x'(t),y'(t))$, while the normal map $\eta_X$ is the map $t \mapsto (-y'(t),x'(t))$. Since $\char(\Bbbk) \neq 2$, we can write the differential maps $(d\tau_X)_p : t \mapsto (x''(t),y''(t))$ and $(d\eta_X)_p : t \mapsto (-y''(t),x''(t))$. Clearly, $(d\tau_X)_p$ is injective if and only if $(d\eta_X)_p$ is injective.
\end{proof}
Since reflexivity of $X$ is guaranteed by $\char(\Bbbk) = 0$ or $\char(\Bbbk) > d$, we immediately get the following.
\begin{corollary}
	If $\char(\Bbbk) = 0$ or $\char(\Bbbk) > d$, then $\eta_X$ is separable.
\end{corollary}

\begin{lemma}\label{lemma:product_degrees}
	Let $X$ be an integral projective plane curve of degree $d$ and class $\nu$. Assume that $\eta_X$ is separable. Then, $\deg(\eta_X) \cdot \deg(X^\perp) = d + \nu.$
\end{lemma}
\begin{proof}
	By Lemma \ref{lemma:separability}, $\tau_X$ is separable and $\deg(X^\vee) = \nu$. We have to prove that for a general $o \in \bbP V$ there are $d+\nu$ points $p \in X_{\rm reg}$ such that the normal to $T_{X,p}$ passes through $o$. We identify such pairs with the fixed points of the following correspondence on $C \times C$. Let $p \in X_{\rm reg}$. The pencil of lines normal to $\langle o,p \rangle$ has a base point $a_p$. Let $U \subseteq X_{\rm reg}$ be a open subset such that there are $\nu$ points on $X$ whose tangent lines pass through $a_p$ whenever $p \in U$. We call $\gamma(p)$ such divisor of degree $\nu$ on $X$. We extend this correspondence to $C$. This defines the divisor $\Gamma \subseteq C \times C$ given by the closure of 
	\[
		\gamma = \{(p,x) ~:~ p \in U, x \in U, T_{X,x} \perp \langle o,p \rangle\}.
	\]
	Let $E$ and $F$ be the fibers of the two projections of $C \times C$ on each factor. We compute their intersections with the divisor $\Gamma$.
	\begin{itemize}
		\item $\#(\Gamma \cap E) = \nu$. As we said, fixed a general point $p \in X_{\rm reg}$, then $\gamma(p)$ has degree $\nu$.
		\item $\#(\Gamma \cap F) = d$. Fix $x \in X_{\rm reg}$ such that $(p,x) \in \gamma$. Let $\{p_1,\ldots,p_d\} = X \cap \langle o,p \rangle$. Then, from the construction above we have that $a_p = a_{p_i}$ for all $i \in \{1,\ldots,d\}$. Hence, $(p_i,x) \in \gamma$ for all $i \in \{1,\ldots,d\}$, i.e., the number of points on $X$ for which line connecting to $o$ is normal to $X$ in $x$ is equal to $d$.
	\end{itemize}
	If $\Delta$ is the diagonal in $C \times C$, then $\Gamma \sim aE + bF - k\Delta$ where the number $k$ is called \textit{valence} of $\Gamma$ (see \cite[page 284]{gh}). In this case, we have that $k = 0$. Indeed, for a general $p \in U$, the divisor $\gamma(p)$ is independent of $p$ because such divisor is cut out on $X$, outside the contribution of the singularities of $X$, by the polar curve of $X$ with respect to the point $a_p$. Hence, $\Gamma \sim \nu E + d F$ and the fixed points of the correspondence are $\#(\Gamma \cap \Delta) = d + \nu$. This concludes the proof.
\end{proof}

In characteristic $0$ the \textit{separable degree} of $\eta_X$ is always equal to the degree of $\eta_X$. If $\char(\Bbbk) = p > 0$, then $\deg \eta_X = p^ct$ where $t$ is the separable degree and $p^c$ is the \textit{inseparable degree}. The separable degree counts the elements in a general fiber of the map, i.e., the number of points at which a general normal line is normal to the curve. Note that Lemma \ref{lemma:separability} tells us that the normal map $\eta_X$ has separable degree equal to its degree if and only if $X$ is reflexive. 

\subsection{Strange curves and their normal lines}\label{ssec:strange}
Recall the definition of \textit{strange curve}.
\begin{definition}[{\cite[page 311]{hart}}]
	A curve $X$ in projective space is \textbf{strange} if there is a point $o$ which lies on all tangent lines of $X$. The point $o$ is called \textbf{strange point} of $X$.
\end{definition}
Assuming that $X$ is not a line, then if a strange point for $X$ exists, then it is unique. It is well-known that strange curves exist only in positive characteristic and the only smooth strange curves are lines and conics in characteristic two \cite{sam,hef89,bh91,hv91}.

The following example produces strange curves with normal map of separable degree $t$, for any $t \geq~1$. 
\begin{example}\label{ex:high_separable}
	Fix $\char(\Bbbk) = p > 2$. Let $e,t \geq 1$ be integers such that $\gcd(p,e) = \gcd(p,t) = 1$. Consider the plane curve $X$ definede by $f = x^ty^{pe}+z^{pe+t}$. The point $o = (0:1:0)$ is the strange point. Indeed,
	\[
		\partial_x(f) = tx^{t-1}y^{pe}, \quad \partial_y(f) = 0, \quad \partial_z(f) = tz^{t-1}.
	\]
	In the chart $\{z = 1\}$, consider a point $(x_0,y_0) \in X$, i.e., $x_0^ty_0^{pe}=-1$. Then, the tangent line at $(x_0,y_0)$ is
	\[
		tx_0^{t-1}y_0^{pe}\cdot x + t = 0 \quad \Rightarrow \quad x = x_0.
	\]
	The line $x = x_0$ meets $X$ at exactly $e$ points since the equation $x_0^ty^{pe} + 1 = 0$ has exactly $e$ roots. At these points the line $x = x_0$ is still tangent. Hence, $\tau_X$ has separable degree equal to $e$. \\
	Assume $Q$ to be the Fermat conic $x^2 + y^2 = 0$ at the line at infinity $z = 0$. Then, the normal line at $(x_0,y_0)$ is the line connecting the point $(x_0:y_0:1)$ and the point $(1:0:0)$, which is the point normal to $(0:1:0)$ with respect to $Q$, i.e., the line $y = y_0$. Now, this line intersects $X$ at exactly $t$ points since the equation $y_0^{pe}x^t+1 = 0$ has exactly $t$ roots. At these points, the line $y=y_0$ is still normal. Hence, $\eta_X$ has separable degree equal to $t$. Moreover, $X^\perp$ is equal to the line of $\bbP V^\vee$ dual to the point $(1:0:0)$. \\ 	
	Note that $t = \deg(f) - pe$ where $pe$ is equal to the multiplicity of the curve $X$ at the point $(1:0:0)$. 
\end{example}
The latter example is a particular case of the following general fact.
\begin{proposition}
	Let $X \subseteq \bbP V_\infty$ be a degree $d \geq 2$ strange curve with strange point $o$. Fix the line at infinity $H_\infty$ and the non-degenerate quadric $Q$ such that $o \in H_\infty$ and $o \not\in Q$. Then, $X^\perp = \ell \subseteq \bbP V^\vee$ is a line. Moreover, if $o_\ell \in \bbP V$ is the point dual to $\ell$, then $\eta_X$ has separable degree $d-\mu$ where $\mu$ is the multiplicity of $X$ at $o_\ell$, with the convension $\mu = 0$ if $o_\ell \not\in X$.
\end{proposition}
\begin{proof}
	Since $X$ is not a line, then $X^\perp$ is a curve. Since $Q$ is non-degenerate, there exists a unique point $o' \in H_\infty$ orthogonal to $o$. Since all tangent lines pass through $o$, then all normal lines pass through $o'$. Hence, $X^\perp$ is a line $\ell$ and $o' =: o_\ell$ is the point dual to it via the identification $\bbP V = \bbP V^{\vee\vee}$. \\
	Note that, since $o \not\in Q$, then $o_\ell \neq o$. Let $\mu$ be the multiplicity of $X$ at $o_\ell$, with the convention $\mu = 0$ if $o_\ell \not\in X$. Now, a general normal line of $X$ cut $X$ in exactly $d-\mu$ distinct points and it is normal to $X$ at such points. In other words, $\eta_X$ has separable degree equal to $d-\mu$.
\end{proof}
\begin{corollary}
	Fix a line $\ell \subseteq \bbP V^\vee$ and let $o_\ell \in \bbP V$ the point dual to $\ell$ via the identification $\bbP V = \bbP V^{\vee\vee}$. Fix the line at infinity $H_\infty$ and the non-degenerate quadric $Q$ such that $o_\ell \in H_\infty$ and $o_\ell \not\in Q$. Then, for any $d \geq \char(\Bbbk) = p > 2$ there is a strange curve such that $X^\perp = \ell$. 
\end{corollary}
\begin{proof}
	Such curves are given by Example \ref{ex:high_separable} where $d = pe+t$ can be any integer larger than $d$. 
\end{proof}
\begin{example}\label{ex:pathology}
	Take $p = \char(\Bbbk) > 2$. Let $X$ and $Y$ be two different strange curves of degree larger than $p$ as the one constructed in Example \ref{ex:high_separable} having the same strange point. Then, from the latter results we have that $X^\perp = Y^\perp$ and they correspond to the line dual to the unique strange point. In the terminology of bottlenecks introduced in Section \ref{ssec:bottlenecks}, these give examples such that $B(X,Y) = X^\perp = Y^\perp$ is a line. 
\end{example}
In the next section we will see that in high characteristic, under the additional assumption of smoothness, the latter pathological example cannot happen.

\section{The proofs}\label{sec:proofs}
\subsection{Birationality of the normal map}
We consider the following notation.

Let $e$ be the separable degree of $\eta_X$. Our goal is to prove that $e = 1$. Let $U \subseteq X_{\rm reg}$ be the open subset such that $N_{X,a}$ meets $X$ at exactly $e$ points. Clearly, $e \leq d = \deg(X)$. 

As before, let $u : C \rightarrow X$ be the normalization of $X$. On $X$, we define the  equivalence relation
\[
	p \sim q \quad \Leftrightarrow \quad N_{X,p} = N_{X,q}.
\]
On the $e$-th symmetric power $S^e(U) = U^{\times e} / \frak{S}_e$, the set of $e$ points in the same equivalence class is an algebraic set. Hence, the equivalence relation defines an algebraic equivalence on the smooth curve $C$ together with a separable morphism $v : C \rightarrow D$ of degree $e$ where $D$ is a connected smooth curve.

Let $|V| = |\calO_C(1)| = |u^*\calO_X(1)|$ be the base point free degree $d$ net of divisors of $C$ which are the pull-backs by $u$ of the intersection of $X$ with the lines of $\bbP V$. Given an intersection $X \cap \ell = \{a_1,\ldots,a_d\}$ for a general line $\ell$, where we allow the case $a_i = a_j$ for some $i \neq j$, then we consider the divisor $E = u^*(a_1 + \ldots + a_d)$ on $C$ and the degree-$d$ divisor $v(E)$ on $D$. Since $v$ is a morphism between connected smooth curves, if two divisors $E_1$ and $E_2$ on $C$ are linear equivalent then also $v(E_1)$ and $v(E_2)$ are linear equivalent. We call $|W|$ the linear system of all degree $d$ linearly equivalent divisors obtained in this way. Hence, $\dim |W| \leq \dim |V| = 2$ and, since $|W|$ is base-point free, then $\dim |W| \geq 1$. 

Before giving the proof of our Theorem \ref{thm:birational}, we observe that the general normal line of a reflexive curve with geometric genus at least $2$ intersects the curve at exactly as many regular points as its degree.
\begin{remark}
	Assume that $X$ is reflexive or, equivalently, that $\eta_X$ is separable. Obviously, if $\deg(X) = d$, then no line is normal to more than $d$ points. In particular, $\eta_X$ has separable degree less or equal than $d$. Since, by Lemma \ref{lemma:product_degrees}, $\deg(\eta_X)\cdot \deg(X^\perp) > d$, we deduce that the curve $X^\perp$ is not a line, i.e., there is no point $o\in \bbP V$ such that every normal line to $X$ passes through the point $o$. 
	
	\textbf{Claim.} \textit{Let $g$ be the geometric genus of $X$. If $g \geq 2$, then $\#(X_{\rm reg} \cap \ell) = d$ for a general normal line $\ell$. }
	
	\begin{proof}[Proof of Claim]
		Fix a general point $p \in X_{\rm reg}$. Since the singular locus ${\rm Sing}(X)$ is a finite set of points, $N_{X,p} \cap {\rm Sing}(X) = \emptyset$. Hence, in order to prove the claim we have to prove that $N_{X,p}$ is not tangent to any point in $X_{\rm reg}$. Assume that $N_{X,p} = T_{X,q}$ for some other point $q \in X_{\rm reg}$. Since $Q$ is a non-degenerate quadric, then $T_{X,p} \neq N_{X,p}$ and therefore $q \neq p$. Thus, we deduce that a general normal line is a tangent line, and viceversa a general tangent line is a normal line. I.e., $X^\perp = X^\vee$ and, since $X$ is reflexivity, $\deg(X^\perp) = \nu$. Now, if $\deg \eta_X = e$, then we deduce that $e\nu = d+\nu$, i.e., $d = (e-1)\nu$. Since $e = \deg \eta_X$ divides $d = \deg X$, then $e = 2$. Observe that: (i) by construction, the smooth curve $D$ is the normalization of $X^\perp$; (ii) since the tangent map $\tau_X$ is birational onto its image and $X^\vee = X^\perp$, we have that $C$ is the normalization of $X^\perp$. Therefore, the curves $C$ and $D$ are isomorphic. Considering the morphism $v : C \rightarrow D$ defined above, from the Riemann-Hurwitz formula, we get that $2g-2 \leq \deg(v) (2g-2)$. Since $\deg(v) = e = 2$, we conclude that $g \leq 1$; contradiction.
	\end{proof}
\end{remark}

\begin{proof}[Proof of Theorem \ref{thm:birational}]
	Let $b$ be a general point of $D$ such that $b = v(a)$ with $a \in U$. 
	
	Let $z$ be the order of contact between $X$ and $T_{X,a}$. Assume for a moment that $z > 2$. Then $\char(\Bbbk) = p$ is a prime, the curve is not reflexive and $z$ is $p$-power \cite{hef89}. Since $z \leq d$, then $p \leq d$; contradiction. Hence, $z = 2$. 
	 
	 We deduce that $\dim |W| = 2$. Indeed, consider the following three divisors in $|W|$:
	 \begin{itemize}
	 	\item $v(E_0)$, where $E_0$ is the pull-back on $C$ of $X \cap \ell$ where $\ell$ is a general line not passing through $a$;
	 	\item $v(E_1)$, where $E_1$ is the pull-back on $C$ of $X \cap \ell$ where $\ell$ is a general line passing through $a$;
	 	\item $v(E_2)$, where $E_2$ is the pull-back on $C$ of $X \cap T_{X,a}$.
	 \end{itemize}
	We note that the three divisors have order $0,1$ and $2$ at the point $b$, respectively. Hence, they are independent and show that $\dim |W| = 2$.
	 
	 Let $E_e$ be the pull-back via $u$ of $X \cap N_{X,a}$. Then, $v(E_e)$ has order $e$ at the point $b$.
	  
	We assume by contradiction that $e \geq 2$:
	\begin{itemize}
		\item if $e \geq 3$, then $v(E_0), v(E_1), v(E_2), v(E_e)$ have three different orders at the point $b$ and therefore imply that $\dim |W| \geq 3$; contradiction;
		\item if $e = 2$ and $v(E_2) \neq v(E_e)$, then there is a linear combination of $E_2$ and $E_e$ which has order higher than two in $b$, implying again that $\dim |W| \geq 3$; contradiction;
		\item if $e = 2$ and $v(E_2) = v(E_e)$ then the surjective map
		\[v_* : |V| \rightarrow |W|, \quad E \mapsto v(E)\]
		is not injective because $E_2 \neq E_e$, i.e., $\dim |W| < 2$; contradiction.
	\end{itemize}
	 This implies $e = 1$ and concludes the proof.
\end{proof}

\subsection{Smooth curves with same curve of normal lines are equal}
We have seen in Section \ref{ssec:strange} examples in positive characteristic of pairs of strange curves $X \neq Y$ such that $X^\perp = Y^\perp$. We prove that this is not the case if we consider smooth plane curves of degree larger than four in characteristic zero or strictly smaller than the degrees.
\begin{proof}[Proof of Theorem \ref{thm:equal_curves}]
	Since $\char(\Bbbk) = 0$ or $\char(\Bbbk) \geq \max\{\deg(X),\deg(Y)\}$, then both $X$ and $Y$ are reflexive and, by Lemma \ref{lemma:separability}, $\eta_X$ and $\eta_Y$ are both separable. Theorem \ref{thm:birational} tells us that $\deg(\eta_X) = \deg(\eta_Y) = 1$. Therefore, by Lemma \ref{lemma:product_degrees}, $\deg X^\perp = \deg(X)^2$ and $\deg Y^\perp = \deg(Y)^2$. Hence, we deduce that $\deg(X) = \deg(Y)$.
	
	Now, note that since $X^\perp = Y^\perp$ and both $\eta_X$ and $\eta_Y$ have degree one, then there is a open subset $U \subseteq X_{\rm reg}$ such that for every $a \in U$ there is a unique point $f(a) \in Y_{\rm reg}$ such that $N_{X,a} = N_{Y,f(a)}$. Moreover, $f$ induces an isomorphism between $U \subseteq X_{\rm reg}$ and an open set $U' \subseteq Y_{\rm reg}$. By smoothness, such isomorphism extends to an isomorphism $f : X \rightarrow Y$.
	
	Since $\deg(X) = \deg(Y) \geq 4$, by \cite[Example 18(iii) at pg. 56]{acgh}, $|\calO_X(1)|$ is the unique net of degree-$d$ divisors on $X$ and analogously $|\calO_Y(1)|$ is the unique net of degree-$d$ divisors on $Y$. In other words, the isomorphism $f : X \rightarrow Y$ is uniquely determined by an automorphism $h \in {\rm Aut}(\bbP V)$. Let $h^\vee \in {\rm Aut}(\bbP V^\vee)$ be the dual automorphism. Since $h(X) = Y$ and $X^\perp = Y^\perp$, then $h^\vee$ defines an automorphism on $X^\perp$. Since $\eta_X : X \rightarrow X^\perp$ is birational by Theorem \ref{thm:birational} and $X$ is smooth, then $X$ is the normalization of $X^\perp$. Thus, $h^\vee$ induces an automorphism on $X$ which will be again induced by a unique automorphism $h_1 \in {\rm Aut}(\bbP V)$. 
	
	Now, the automorphisms $h^\vee$ and $h_1^\vee$ on $\bbP V$ induce the same map on the curve $X^\perp$, i.e., $h^\vee = h_1^\vee$. Hence,
	\[
		h_1 = h_1^{\vee\vee} = h^{\vee\vee} = h.
	\]
	In particular,
	\[
		Y = h(X) = h_1(X) = X. \qedhere
	\] 
\end{proof}

\end{document}